\documentclass[11pt,a4paper]{article}
\usepackage{amsmath, amssymb, amsthm}
\usepackage{enumitem}
\usepackage{geometry}
\usepackage{diagbox}
\usepackage{boondox-cal}
\usepackage{graphicx}

\newtheorem{theorem}{Theorem}[section]
\newtheorem{proposition}[theorem]{Proposition}

\theoremstyle{definition}
\newtheorem{definition}[theorem]{Definition}

\newtheorem{remark}[theorem]{Remark}

\newcommand\blfootnote[1]{%
  \begingroup
  \renewcommand\thefootnote{}\footnote{#1}%
  \addtocounter{footnote}{-1}%
  \endgroup
}

\newcommand{\ecal}{\mathord{\ensuremath{\scalebox{1.2}{$\mathcal{e}$}}}}
\newcommand{\fcal}{\mathord{\ensuremath{\scalebox{1.2}{$\mathcal{f}$}}}}
\newcommand{\zcal}{\mathord{\ensuremath{\scalebox{1.2}{$\mathcal{z}$}}}}
\newcommand{\ucal}{\mathord{\ensuremath{\scalebox{1.2}{$\mathcal{u}$}}}}

\title{Associative Structures in Pseudo-Riemannian Lie Algebras}
{\author{
  \begin{tabular}{@{}c@{}}
Santiago \textsc{Casta\~{n}eda-Montoya}\\
Edison Alberto \textsc{Fern\'andez-Culma} \textsuperscript{$\dagger$}\blfootnote{${}^{\dagger}$ Partially supported by CONICET and SECyT-UNC: 33620180100523CB} \\
    {\small Centro de Investigaci\'on y Estudios en Matem\'atica de C\'ordoba, CONICET}\\
    {\small Facultad de Matem\'atica, Astronom\'ia, F\'isica y Computaci\'on, UNC}\\
    {\small C\'ordoba, Argentina}\\
    \texttt{\small santiago.castaneda@mi.unc.edu.ar}\\
    \texttt{\small efernandez@famaf.unc.edu.ar}\\
  \end{tabular}
}}
\date{May 24, 2026}

\begin{document}

\maketitle

\blfootnote{2020 \textit{Mathematics Subject Classification:} {53C50; 22E15; 53C22; 17B60}. \textit{Key words and phrases:} Pseudo-Riemannian Lie algebras; Associative $U$-tensor; Geodesic completeness; Almost-abelian Lie algebras; Two-step nilpotent Lie algebras}


\begin{abstract}
This paper investigates the algebraic and geometric consequences of the associativity of the symmetric part $U$ of the Levi-Civita connection on a pseudo-Riemannian Lie algebra $(\mathfrak{g}, \langle \cdot, \cdot \rangle)$. We demonstrate that in the Riemannian (positive-definite) setting, the associativity of $U$ is an extremely restrictive condition that forces the tensor to vanish identically, thereby recovering the class of bi-invariant metrics. In contrast, in the pseudo-Riemannian setting, we focus on the subclass where $(\mathfrak{g}, U)$ is associative and unimodular. As a primary result, we establish that every connected Lie group endowed with a left-invariant pseudo-Riemannian metric whose $U$-tensor is associative and unimodular is geodesically complete. Finally, we explore the families of 2-step nilpotent and almost-abelian Lie algebras. For the latter, we obtain some rigid structural classifications, showing that the paradigmatic models are the $3$-dimensional Heisenberg algebra with certain (anti)-Lorentzian metrics or a semi-direct extension involving a nondegenerate infinitesimal $\beta$-transformation on the canonical neutral space $W \oplus W^*$.

\end{abstract}



\section{Introduction}
The study of left-invariant metrics on Lie groups constitutes a fundamental chapter in modern differential geometry, situated at the intersection of algebraic structures and curvature properties. Since the foundational work of John Milnor \cite{Milnor} on the curvatures of left-invariant metrics, significant effort has been devoted to understanding how the properties of the Lie algebra $(\mathfrak{g}, [\cdot, \cdot])$ determine the global geometry of the associated Lie group $G$.

For a Lie group $G$ endowed with a pseudo-Riemannian metric $g$, let $(\mathfrak{g}, \langle \cdot, \cdot \rangle)$ be the corresponding pseudo-Riemannian Lie algebra. A central tool in the analysis of its geometry is the Levi-Civita connection $\nabla$, which can be decomposed into a skew-symmetric part (the Lie bracket) and a symmetric bilinear tensor $U: \mathfrak{g} \times \mathfrak{g} \to \mathfrak{g}$ defined by the Koszul identity: $2\langle U(X, Y), Z \rangle = \langle [Z, X], Y \rangle + \langle [Z, Y], X \rangle$. The $U$-tensor encodes critical information regarding the geometry of $(G,g)$; in particular, its vanishing ($U=0$) characterizes the class of bi-invariant metrics (when $G$ is connected).

In recent years, several authors have explored the geometric implications of imposing specific algebraic identities on the Levi-Civita product. For instance Novikov or Leibniz structures impose strong constraints on the holonomy and local symmetry of the space \cite{BenayadiBoucetta}. A prominent example is the study of Left Symmetric Algebras (LSA), which provide the necessary algebraic framework for studying Lie groups endowed with a left-invariant, torsion-free, and flat affine connection. This line of research has recently reached a definitive milestone with the complete structural description of flat Lorentzian Lie groups provided by Mohamed Boucetta in \cite{Boucetta}.

In this paper, we investigate the associativity of $U$ as a natural algebraic extension of the bi-invariant setting. We demonstrate a rigidity theorem for the Riemannian (positive-definite) case: the associativity of $U$ necessarily forces $U=0$, thereby recovering the class of bi-invariant metrics. In contrast, the pseudo-Riemannian setting admits a much richer variety of non-trivial structures due to the existence of isotropic directions. We prove that the associativity of $U$, combined with the cyclic symmetry of the tensor $\theta(X,Y,Z)=\langle U(X,Y),Z\rangle$ and a natural unimodularity hypothesis on $U$, makes $(\mathfrak{g}, U)$ a nilpotent algebra.

We argue that these conditions represent a natural algebraic proximity to the bi-invariant case; whereas bi-invariance requires $U$ to be zero (the trivial abelian case), our hypotheses allow $U$ to be nilpotent (the closest structural regime to the abelian setting). This property has profound consequences for the theory of geodesics. While geodesic completeness is a complex and generally open problem in pseudo-Riemannian geometry, the associativity and unimodularity of $U$ allow the reduced geodesic equation $\dot{X} + U(X, X) = 0$ to be integrated as a truncated power series. Following the techniques in \cite{Guediri}, we show that this ensures the geodesic flow is polynomial, thereby guaranteeing the global completeness of the metric.

The paper concludes with a structural analysis of almost-abelian and 2-step nilpotent Lie algebras. For the almost-abelian case, we obtain certain rigid classification results, demonstrating that there is a family which is essentially built from a set of building blocks: specifically, the 3-dimensional Heisenberg algebra with an isotropic center and certain infinitesimal $\beta$-field extensions on canonical neutral spaces. Our results show how the geometry of isotropic subspaces completely determines the viability of these associative structures within this family. In the family of $2$-step nilpotent Lie algebras, we identify key algebraic constraints that prevent associativity when the metric restriction to the center is non-degenerate.

\section{Preliminaries}

Let $G$ be a Lie group and let $\mathfrak{g}$ denote its Lie algebra, which is the finite-dimensional real vector space of left-invariant smooth vector fields on $G$.  If $g$ is a \textit{left invariant
pseudo-Riemannian metric} on $G$, i.e. the left translations are isometries of $(G,g)$, then $g$ is completely determined by the the scalar product $\langle \cdot , \cdot \rangle$ on $\mathfrak{g}$ induced by $g$:
\begin{eqnarray*}
  \langle X , Y \rangle &= & g(X,Y), \qquad \mbox{for all } X,Y \in \mathfrak{g},
\end{eqnarray*}
We refer to the pair $(\mathfrak{g}, \langle \cdot , \cdot \rangle)$ as a \textit{pseudo-Riemannian Lie algebra}. Conversely, any scalar product on $\mathfrak{g}$ uniquely defines a left-invariant pseudo-Riemannian metric on $G$. For the sake of convenience, we shall assume throughout this article that $G$ is connected.

Two pseudo-Riemannian Lie algebras $(\mathfrak{g},\langle \cdot , \cdot \rangle_{\mathfrak{g}})$ and $(\mathfrak{h},\langle \cdot , \cdot \rangle_{\mathfrak{h}})$ are said to be \textit{isometrically isomorphic} if there exists a Lie algebra isomorphism $\varphi: \mathfrak{g} \rightarrow \mathfrak{h}$ such that $\varphi^{\ast}\langle \cdot , \cdot \rangle_{\mathfrak{h}} = \langle \cdot , \cdot \rangle_{\mathfrak{g}}$. In this case, $\varphi$ is called an \textit{isometric isomorphism}.

It is a standard fact that the Levi-Civita connection $\nabla$ of a left-invariant metric $(G,g)$ is a \textit{left-invariant affine connection}. That is, for any $X, Y \in \mathfrak{g}$, the vector field $\nabla_{X} Y$ is also in $\mathfrak{g}$. Consequently, the restriction of $\nabla$ to $\mathfrak{g}$ is a well-defined bilinear map $\nabla: \mathfrak{g} \times \mathfrak{g} \rightarrow \mathfrak{g}$. This algebraic operator completely determines the Levi-Civita connection on $G$, since $\mathfrak{g}$ generates the $C^{\infty}(G)$-module of smooth vector fields $\mathfrak{X}(G)$.

By the \textit{Koszul formula}, we have:
\begin{equation} \label{koszul}
\langle \nabla_{X}Y , Z \rangle = \frac{1}{2}\big( \langle[X,Y],Z\rangle - \langle [Y,Z],X \rangle + \langle[Z,X],Y\rangle \big),
\end{equation}
for all $X,Y,Z \in \mathfrak{g}$. From Equation \eqref{koszul}, we can decompose $\nabla$ into its skew-symmetric and symmetric parts:
$$
\nabla_{X}Y = \frac{1}{2}[X,Y] + U(X,Y),
$$
where the symmetric bilinear map $U:\mathfrak{g} \times \mathfrak{g} \rightarrow \mathfrak{g}$ is implicitly defined by:
\begin{equation}\label{definicionU}
\langle U(X,Y),Z  \rangle = \frac{1}{2} \langle [Z,X] , Y \rangle + \frac{1}{2} \langle [Z,Y] , X \rangle.
\end{equation}
A distinguished class of pseudo-Riemannian metrics on Lie groups is that of \textit{bi-invariant metrics}, where both left and right translations act as isometries. For a connected Lie group $G$, it is well known that a pseudo-Riemannian metric is bi-invariant if and only if $\operatorname{ad}_X$ is skew-symmetric with respect to $\langle \cdot, \cdot \rangle$ for all $X \in \mathfrak{g}$. Equivalently, the Levi-Civita connection on $\mathfrak{g}$ has a vanishing symmetric part (i.e., $U = 0$).

We are interested in studying the conditions under which the bilinear map $U$ defines an associative product on $\mathfrak{g}$. This property can be regarded as a natural generalization of the bi-invariant case; indeed, while bi-invariant metrics are characterized by $U = 0$ (the abelian case), an associative $U$ leads, under certain hypotheses, to a nilpotent structure. A detailed analysis of this condition and its relationship with the geometry of the Lie group will be provided in the following sections.

It is convenient to introduce the covariant $3$-tensor $\theta$ on $\mathfrak{g}$, defined by $\theta(X,Y,Z)= \langle U(X,Y),Z \rangle$. This tensor measures the obstruction of $\langle \cdot , \cdot \rangle$ from defining a bi-invariant pseudo-Riemannian metric. Since $U$ is symmetric, $\theta$ is symmetric with respect to its first two variables. Another well-known property of $\theta$ is its cyclic identity:
\begin{equation} \label{ciclico}
\theta(X,Y,Z) + \theta(Y,Z,X) + \theta(Z,X,Y) = 0,
\end{equation}
for all $X,Y,Z \in \mathfrak{g}$.

Finally, note that $\theta$ is a \textit{metric invariant}. That is, if $\varphi:\mathfrak{g} \rightarrow \mathfrak{h}$ is an isometric isomorphism between two pseudo-Riemannian Lie algebras $(\mathfrak{g},\langle \cdot , \cdot \rangle_{\mathfrak{g}})$ and $(\mathfrak{h},\langle \cdot , \cdot \rangle_{\mathfrak{h}})$, then their respective $\theta$ tensors are equivalent:
$$
\varphi^{\ast} \theta_{(\mathfrak{h},\langle \cdot , \cdot \rangle_{\mathfrak{h}})} = \theta_{(\mathfrak{g},\langle \cdot , \cdot \rangle_{\mathfrak{g}})}.
$$

To adopt a more algebraic perspective, let us denote by $\divideontimes$ the product on $\mathfrak{g}$ induced by $U$, defined as $X \divideontimes Y := U(X,Y)$ for all $X, Y \in \mathfrak{g}$. A straightforward observation is that if $(\mathfrak{g}, \divideontimes)$ is an associative algebra, then it cannot possess a unit element due to the cyclic identity \eqref{ciclico}. Indeed, if $\ucal$ were a unit for $(\mathfrak{g}, \divideontimes)$, then for any $X \in \mathfrak{g}$ we would have:
\begin{equation*}
\langle \ucal \divideontimes \ucal , X \rangle + \langle \ucal \divideontimes X , \ucal \rangle + \langle X \divideontimes \ucal , \ucal \rangle = 0.
\end{equation*}
Since $\divideontimes$ is commutative and $\ucal$ is the unit element, the above expression simplifies to $3\langle \ucal, X \rangle = 0$ for all $X \in \mathfrak{g}$. Since the scalar product $\langle \cdot, \cdot \rangle$ is nondegenerate, this implies $\ucal = 0$, which contradicts the definition of a unit in a non-trivial algebra.

\section{Definite positive case.}

The study of left-invariant Riemannian metrics on Lie groups is, by far, the most extensively developed case in the literature. One of the foundational works in this direction, which motivated several lines of research on the geometry of Lie groups, is the seminal paper by John Milnor \cite{Milnor}. The first structural results concerning metric Lie algebras that satisfy specific curvature conditions (such as being flat, satisfying pinching conditions, or being Einstein) were obtained within this framework.

Assuming the left-invariant metric to be positive-definite, we prove that the $U$-tensor is associative only if the metric is bi-invariant.

\begin{theorem}
Let $(\mathfrak{g}, \langle \cdot, \cdot \rangle)$ be a metric Lie algebra endowed with a positive-definite inner product. If the product $\divideontimes$, defined by the symmetric part of the Levi-Civita connection of $(\mathfrak{g}, \langle \cdot, \cdot \rangle)$, is associative, then this product vanishes identically, which implies that the metric is bi-invariant.
\end{theorem}

\begin{proof}
Since $(\mathfrak{g}, \divideontimes)$ is a commutative and associative algebra, the powers of any element $X \in \mathfrak{g}$ with respect to the product $\divideontimes$ are well-defined. Therefore, we can apply the cyclic condition (\ref{ciclico}), which yields that for any $l,m,n \in \mathbb{N}$ and any $X \in \mathfrak{g}$,
\begin{eqnarray} \label{ciclico2}
\langle X^{l+m}, X^{n} \rangle + \langle X^{m+n}, X^{l} \rangle + \langle X^{n+l}, X^{m} \rangle = 0.
\end{eqnarray}
This relation can be viewed as an algebraic expression of degree $m+n+l$. By fixing the degree of the expression, we obtain a system of homogeneous equations involving the inner product and the powers induced by $\divideontimes$.

Let us first consider expressions of degree 6. Setting $l=m=n=2$, equation (\ref{ciclico2}) gives:
$$
3\langle X^4, X^2 \rangle = 0.
$$
Next, setting $l=m=1$ and $n=4$, we obtain:
$$
\langle X^2, X^4 \rangle + 2 \langle X^5, X^1 \rangle = 0.
$$
Given our previous deduction, this implies that $\langle X^5, X^1 \rangle = 0$. Now, substituting $l=1$, $m=2$, and $n=3$ into relation (\ref{ciclico2}) yields:
$$
\langle X^3, X^3 \rangle + \langle X^5, X^1 \rangle + \langle X^4, X^2 \rangle = 0.
$$
Since the last two terms vanish, we get $\langle X^3, X^3 \rangle = 0$. Because the metric $\langle \cdot, \cdot \rangle$ is positive-definite, this forces $X^3 = 0$ for all $X \in \mathfrak{g}$.

We now analyze expressions of degree 4. Setting $l=m=1$ and $n=2$, relation (\ref{ciclico2}) becomes:
$$
\langle X^2, X^2 \rangle + 2\langle X^3, X^1 \rangle = 0.
$$
Since $X^3 = 0$, the second term vanishes, resulting in $\langle X^2, X^2 \rangle = 0$. Again, by the positive-definiteness of the metric, this implies $X^2 = 0$ for all $X \in \mathfrak{g}$.

Finally, because $(\mathfrak{g}, \divideontimes)$ is a commutative algebra, the \textit{linearization/polarization} of the identity $X^2 = 0$ implies that the product $\divideontimes$ vanishes identically for all elements ($\operatorname{char}(\mathbb{R}) \neq 2$). The vanishing of the symmetric part of the Levi-Civita connection precisely means the metric is bi-invariant, which concludes the proof.
\end{proof}

\section{Geometric consequences of $U$-tensor associativity}

In this section, we focus on a specific class of metrics characterized by the property that the $U$-tensor induces an associative and unimodular algebra structure on $\mathfrak{g}$. In this context, unimodularity is understood in the sense that the left-multiplication operators $L_X: Y \mapsto X \divideontimes Y$ satisfy $\operatorname{tr}(L_X) = 0$ for all $X \in \mathfrak{g}$. This algebraic restriction is naturally motivated by the following analysis of a tensor space where the tensor $\theta$ resides.

\subsection{Weyl Decomposition of the space of tensors with the same symmetries of $\theta$.}

We briefly review the structural decomposition of the space $\mathcal{W} \subset \otimes^3 V^*$ consisting of covariant $3$-tensors that satisfy the symmetry and cyclic conditions previously identified under the action of pseudo-orthogonal groups. While the underlying representation theory is a classical consequence of Weyl's foundational work and belongs to the established mathematical folklore of invariant theory, we detail it here to ensure the self-containment of our exposition and to fix the algebraic notation required for the subsequent sections.

Let $V$ be an $n$-dimensional real vector space equipped with a scalar product $\langle\cdot,\cdot \rangle$ of signature $(p,q)$. We denote its associated orthogonal group as $O(V, \langle\cdot,\cdot \rangle)$. We consider the space $\mathcal{W} \subset \otimes^3 V^*$ defined by the following algebraic constraints:
\begin{enumerate}
    \item \textbf{Partial Symmetry:} $\theta(x,y,z) = \theta(y,x,z)$, i.e., $\theta \in S^2 V^* \otimes V^*$.
    \item \textbf{Cyclic Identity:} $\theta(x,y,z) + \theta(y,z,x) + \theta(z,x,y) = 0$.
\end{enumerate}

To decompose this space into irreducible components under the action of $O(V,\langle\cdot,\cdot \rangle)$, we rely on the fundamental results of Hermann Weyl regarding the representation theory of the classical groups \cite{Weyl}
(see also \cite[Chapters 9, 10]{GoodmanWallach}).

\begin{proposition}[{\cite {Weyl}}]
Let $\mathcal{W}$ be the space of covariant $3$-tensor  defined above, with $\operatorname{dim}(V)\geq 3$. The structural properties of this space under the action of the nested groups $O(V,\langle\cdot,\cdot \rangle) \subset GL(V)$ are characterized as follows:
\begin{enumerate}
    \item \textbf{General Linear Action:} Under the natural action of $GL(V)$, the space $\mathcal{W}$ is an irreducible $GL(V)$-module corresponding to the partition $\lambda = (2,1)$. Its dimension is given by:
    \begin{equation}
        \dim(\mathcal{W}) = \frac{n(n^2 - 1)}{3}.
    \end{equation}
    \item \textbf{Orthogonal Decomposition:} Under the restricted action of the orthogonal group $O(V,\langle\cdot,\cdot \rangle)$, $\mathcal{W}$ decomposes into two irreducible $O(V,\langle\cdot,\cdot \rangle)$-modules:
    \begin{equation}
        \mathcal{W} = \mathcal{W}_0 \oplus \mathcal{V}
    \end{equation}
    where $\mathcal{W}_0$ is the \textbf{traceless component}; the kernel of the metric contraction $\text{Tr}_{1,2}^{\langle\cdot,\cdot \rangle}$, and $\mathcal{V} \cong V^* \cong V$ is the \textbf{vector component}.
\end{enumerate}
\end{proposition}

\begin{proof}$ $
\begin{enumerate}
\item Irreducibility under $GL(V)$: \\
By the well-known \textit{Schur-Weyl duality}, the algebraic constraints (1) and (2) identify $\mathcal{W}$ as the image of the Young symmetrizer associated with the partition $\lambda = (2,1)$. As established in Weyl's treatise on the representations of the general linear group \cite[Chapter IV]{Weyl}, this duality ensures a one-to-one correspondence between the irreducible representations of the symmetric group $\mathfrak{S}_k$ and the finite-dimensional irreducible representations of $GL(V)$ that appear in the decomposition of $(V^\ast)^{\otimes k}$. Consequently, from \cite[Th. 4.4.D]{Weyl}, it follows that $\mathcal{W}$ is an irreducible $GL(V)$-module. Its dimension is computed via the Weyl's dimension formula:
\[ \dim(\mathcal{W}) = \frac{n(n^2-1)}{3}. \]

\item[2.1] Decomposition under $O(V,\langle\cdot,\cdot \rangle)$: \\
When the group action is restricted to $O(V,\langle\cdot,\cdot \rangle)$, we follow \cite[Chapter V, Section 6]{Weyl}. Consider the trace map $\text{Tr}^{\langle\cdot,\cdot\rangle}_{1,2} : \mathcal{W} \to V^*$ defined by
\[ \text{Tr}^{\langle\cdot,\cdot\rangle}_{1,2}(\theta)(z) = \sum_{i,j=1}^{n} g^{ij} \theta(e_i, e_j, z), \]
where $\{e_i\}$ is a basis of $V$ and $(g^{ij})$ is the inverse matrix of the metric coefficients $g_{ij} = \langle e_i, e_j \rangle$. This map is well-defined and $O(V,\langle\cdot,\cdot \rangle)$-equivariant. Since $V^*$ is irreducible for $O(V,\langle\cdot,\cdot \rangle)$, we have a split exact sequence of $O(V,\langle\cdot,\cdot \rangle)$-modules:
\[ 0 \longrightarrow \mathcal{W}_0 \xrightarrow{\ \iota_0 \ } \mathcal{W} \xrightarrow{\ \text{Tr}^{\langle\cdot,\cdot\rangle}_{1,2} \ } V^* \longrightarrow 0 \]
where $\iota_0$ denotes the natural inclusion of the traceless kernel $\mathcal{W}_0 = \ker( \text{Tr}^{\langle\cdot,\cdot\rangle}_{1,2} )$ into $\mathcal{W}$. To prove that the sequence splits, we introduce the equivariant embedding $\iota: V^* \to \mathcal{W}$ such that $( \text{Tr}^{\langle\cdot,\cdot\rangle}_{1,2} \circ \iota) = \text{Id}_{V^*}$ (up to a non-zero constant), thereby doing $\mathcal{V} = \text{im}(\iota)$ as a direct summand. For $\alpha \in V^*$, we define:
\[ \iota(\alpha)(x,y,z) = 2\langle x,y \rangle \alpha(z) - \langle x,z\rangle \alpha(y) - \langle y,z\rangle \alpha(x). \]
A direct calculation verifies that $\iota(\alpha)$ satisfies both the partial symmetry and the cyclic identity. Furthermore, evaluating its trace yields:
\[ \text{Tr}^{\langle\cdot,\cdot\rangle}_{1,2}(\iota(\alpha))(z) = 2n\alpha(z) - \alpha(z) - \alpha(z)= 2(n-1)\alpha(z). \]
Since $n > 1$, this trace is non-zero, proving that $\mathcal{W} \cong \mathcal{W}_0 \oplus \mathcal{V}$.

\item[2.2] Irreducibility of $\mathcal{W}_0$: \\
Although Weyl's original formulation emphasizes the compact orthogonal group $O(n)$, by the Weyl Unitary Trick \cite[Ch. VIII, \S11]{Weyl}, the finite-dimensional representation theory of the pseudo-orthogonal
group $O(p,q)$ is completely reducible, and its irreducible modules are in natural bijection with those of $O(n)$ via their common complexification $O(n,\mathbb{C})$. The irreducibility of $\mathcal{W}_0$ then follows from
\cite[Ch. V, \S7, Theorem (5.7.H)]{Weyl}: since the index symmetries of $\mathcal{W}$ imply that $\mathrm{Tr}^{\langle\cdot,\cdot\rangle}_{1,3}$ and $\mathrm{Tr}^{\langle\cdot,\cdot\rangle}_{2,3}$ are proportional to
$\mathrm{Tr}^{\langle\cdot,\cdot\rangle}_{1,2}$, the kernel $\mathcal{W}_0 = \ker(\mathrm{Tr}^{\langle\cdot,\cdot\rangle}_{1,2})$ coincides with the vanishing of \emph{all} metric contractions on
$\mathcal{W}$, and is therefore precisely the traceless subspace $P_0(\lambda)$ in the sense of Weyl. The diagram $\lambda = (2,1)$ is permissible since the total length of its first two columns satisfies $2 + 1 = 3 \leq \operatorname{dim}(V)$, so Theorem (5.7.H) applies and $\mathcal{W}_0$ is an irreducible $O(V,\langle\cdot,\cdot\rangle)$-module.
\end{enumerate}
\end{proof}

\subsection{Geometry of $\mathcal{W}_0$}

\begin{definition}
Let $(\mathfrak{g}, \langle \cdot, \cdot \rangle)$ be a pseudo-Riemannian Lie algebra. We say that $(\mathfrak{g}, \langle \cdot, \cdot \rangle)$ belongs to the class $\mathcal{W}_0$ if its associated tensor $\theta_{(\mathfrak{g}, \langle \cdot, \cdot \rangle)}$ lies in $\mathcal{W}_0$.
\end{definition}

It is worth noting that the associativity condition on the $U$-tensor defines an \textit{algebraic set} within the vector space $\mathcal{W}$. Rather than tackling this problem in full generality, we restrict our focus to the subclass $\mathcal{W}_0$. Due to the symmetries of the $\theta$ tensor, the following identity holds: $2 \operatorname{Tr}^{\langle\cdot,\cdot\rangle}_{2,3} \theta = -\operatorname{Tr}^{\langle\cdot,\cdot\rangle}_{1,2} \theta$. Consequently, if $(\mathfrak{g}, \langle \cdot, \cdot \rangle)$ belongs to the class $\mathcal{W}_0$, the induced algebra $(\mathfrak{g}, \divideontimes)$ is \textit{unimodular}; that is, the left-multiplication operators $L_X: \mathfrak{g} \rightarrow \mathfrak{g}$ defined by $Y \mapsto X \divideontimes Y$ satisfy $\operatorname{tr}(L_X) = 0$ for all $X \in \mathfrak{g}$.

\begin{remark}
From Equation (\ref{definicionU}), since $\operatorname{Tr}(L_x) = -\operatorname{ad}_X$, it follows that $(\mathfrak{g},[\cdot,\cdot])$ is a unimodular Lie algebra if and only if $(\mathfrak{g},\divideontimes)$ is also unimodular.
\end{remark}

\begin{proposition}\label{claseW0}
Let $(\mathfrak{g}, \langle \cdot, \cdot \rangle)$ be a pseudo-Riemannian Lie algebra in the class $\mathcal{W}_0$. If the $U$-tensor is associative, then the product defined by $U$ endows $\mathfrak{g}$ with the structure of a nilpotent commutative associative algebra.
\end{proposition}

\begin{proof}
The proof follows directly from the associativity of the product $\divideontimes$. For any $X \in \mathfrak{g}$ and $n \in \mathbb{N}$, associativity implies that the $n$-th power of the left-multiplication operator satisfies $(L_X)^n = L_{X^n}$. By the unimodularity hypothesis, $\operatorname{tr}(L_{X^n}) = 0$ for all $n \ge 1$. Over a field of characteristic zero, the vanishing of the traces of all powers of an endomorphism implies that the endomorphism is nilpotent; thus, $L_X$ is a nilpotent transformation for every $X$. Furthermore, since the algebra $(\mathfrak{g}, \divideontimes)$ is commutative, it follows that the family of nilpotent operators $\{L_X : X \in \mathfrak{g}\}$ is a commuting family. By standard results in linear algebra (specifically, that commuting nilpotent operators are simultaneously triangularizable) it follows that there exists a basis in which all $L_X$ are strictly upper triangular. This implies that the algebra $(\mathfrak{g}, \divideontimes)$ is nilpotent, as required.
\end{proof}

\subsubsection{Geodesics}

The study of geodesics in Lie groups endowed with left-invariant metrics is a classical topic in differential geometry. While it is well known that every left-invariant positive-definite metric is geodesically complete, the indefinite case is vastly more intricate and completely uncertain in general. Much of our understanding in the pseudo-Riemannian setting stems from the foundational work of Mohammed Guediri \cite{Guediri}. For the sake of completeness, we briefly recall the geodesic equation for left-invariant metrics.

Let $\gamma: I \subseteq \mathbb{R} \rightarrow G$ be a smooth curve. At each point $\gamma(t)$, its velocity vector $\dot{\gamma}(t)$ belongs to the tangent space $T_{\gamma(t)}G$. Using a basis of left-invariant vector fields $\{E_1, \dots, E_n\}$ of $\mathfrak{g}$, we can express this velocity as:
$$
\dot{\gamma}(t) = \sum_{k=1}^n x_{k}(t) E_k(\gamma(t)).
$$
The curve $\gamma$ is a geodesic if the covariant derivative of its velocity field along the curve vanishes:
$$
\frac{D}{dt} \dot{\gamma}(t) = 0.
$$
Using the standard properties of the covariant derivative along a curve, this yields the well-known geodesic equation in terms of the basis components:
\begin{eqnarray}\label{geodesica}
0 = \sum_{k=1}^n x^{'}_{k}(t) E_k(\gamma(t)) + \sum_{k=1}^n \sum_{j=1}^n  x_{k}(t) x_{j}(t) \nabla_{E_j} E_k (\gamma(t)).
\end{eqnarray}
This expression naturally motivates the definition of a time-dependent vector in the Lie algebra $\mathfrak{g}$, given by $\displaystyle X(t) = \sum_{k=1}^n x_{k}(t) E_k$, and the corresponding differential equation defined purely on $\mathfrak{g}$:
$$
\dot{X} + \nabla_{X}X = 0.
$$
By evaluating along the curve $\gamma(t)$, we retrieve Equation (\ref{geodesica}). Recalling that $\nabla_X X = \frac{1}{2}[X, X] + U(X, X)$, we arrive at the fundamental reduced geodesic equation:
\begin{eqnarray}\label{geodesica2}
\dot{X} + U(X, X) = 0.
\end{eqnarray}
The problem is thus decoupled into two stages. One first solves the quadratic ordinary differential equation (\ref{geodesica2}) in the vector space $\mathfrak{g}$. Once the solution $X(t)$ is found, the geodesic is reconstructed by integrating the equation $\dot{\gamma}(t) = X(t)_{\gamma(t)}$ on the group $G$. Consequently, a left-invariant metric is geodesically complete if and only if all solutions to $\dot{X} = - X \divideontimes X$ are defined for all $t \in \mathbb{R}$.

If $(\mathfrak{g}, \langle \cdot, \cdot \rangle)$ is a pseudo-Riemannian Lie algebra in the class $\mathcal{W}_0$ and its $U$-tensor is associative, a remarkable consequence of Proposition \ref{claseW0} is that the metric is always complete. To see this, we can propose a formal power series solution for Equation (\ref{geodesica2}),
$$
X(t) = \sum_{k=0}^{\infty} A_k t^k, \quad \mbox{with } A_k \in \mathfrak{g}.
$$
Given that the algebra $(\mathfrak{g}, \divideontimes)$ is commutative and assumed to be associative, substituting this series into the differential equation requires the coefficients to satisfy $A_k = (-1)^k X_0^{k+1}$, yielding:
$$
X(t) = \sum_{k=0}^{\infty} (-t)^k X_0^{k+1},
$$
where $X_0 = X(0)$. By Proposition \ref{claseW0}, the algebra $(\mathfrak{g}, \divideontimes)$ is nilpotent. Therefore, there exists an integer $N \in \mathbb{N}$ such that $X_0^{N+1} = 0$. The formal power series naturally truncates into a finite sum:
$$
X(t) = \sum_{k=0}^{N} (-t)^k X_0^{k+1}.
$$
Since this is a polynomial in $t$, the solution $X(t)$ is globally defined for all $t \in \mathbb{R}$. This guarantees the completeness of the metric, proving the following theorem:

\begin{theorem}
Let $(\mathfrak{g}, \langle \cdot, \cdot \rangle)$ be a pseudo-Riemannian Lie algebra in the class $\mathcal{W}_0$. If the $U$-tensor is associative, then the corresponding left-invariant pseudo-Riemannian metric on $G$ is geodesically complete.
\end{theorem}

\begin{remark}
An interesting consequence of Theorem \ref{claseW0} is that $(G, g)$ admits geodesics that are one-parameter subgroups of $G$. Indeed, since the associative algebra $(\mathfrak{g}, \divideontimes)$ is nilpotent, its \textit{annihilator}
$$
\operatorname{Ann}(\mathfrak{g}, \divideontimes) = \{ X \in \mathfrak{g} : X \divideontimes Y = 0, \mbox{ for all } Y \in \mathfrak{g} \}
$$
is necessarily non-trivial. For any non-zero element $X \in \operatorname{Ann}(\mathfrak{g})$ (or, more generally, any nonzero nilpotent element of index two), the condition $\nabla_X X = X \divideontimes X = 0$ is satisfied. Consequently, the curve $\gamma: \mathbb{R} \to G$ defined by $\gamma(t) = \exp(tX)$ is the unique geodesic starting at the identity $e \in G$ with initial velocity $X_e$.
\end{remark}

\section{Example: Almost-abelian Lie algebras.}\label{almostabelian}

Let $\mathfrak{g}$ be an almost-abelian Lie algebra; that is, a non-abelian Lie algebra containing an abelian ideal of codimension one. It is straightforward to show that $\mathfrak{g}$ possesses a unique codimension-one ideal if and only if it is not isomorphic to the direct product $\mathfrak{h}_{3}(\mathbb{R}) \times \mathbb{R}^{k}$ (this can be deduced from the \textit{correspondence theorem} for Lie algebras together with the fact that $\mathfrak{aff}(\mathbb{R})$ has a unique codimension-one ideal). Here, $\mathfrak{h}_{3}(\mathbb{R})$ denotes the $3$-dimensional real Heisenberg algebra spanned by $\{E, F, Z\}$ with the single non-zero bracket $[E,F]=Z$.

Let $(\mathfrak{g}, \langle \cdot, \cdot \rangle)$ be a pseudo-Riemannian Lie algebra and assume that $\mathfrak{g}$ possesses a non-degenerate abelian ideal $\mathfrak{a}$ of codimension one. It follows that $\mathfrak{g}$ admits the orthogonal direct sum decomposition $\mathfrak{g} = \mathfrak{a} \oplus \mathfrak{a}^{\perp}$ (see \cite[Ch. 2, Lemma 23]{ONeill}). Since $\dim(\mathfrak{a}^{\perp}) = 1$, we may choose a unit vector $\ecal$ spanning $\mathfrak{a}^{\perp}$, such that $\langle \ecal , \ecal \rangle = \varepsilon \in \{1, -1\}$. Furthermore, assuming that the $U$-tensor of $(\mathfrak{g}, \langle \cdot, \cdot \rangle)$ is associative, we aim to characterize the resulting structure of $\mathfrak{g}$ and the metric $\langle \cdot, \cdot \rangle$.

Since $\mathfrak{a}$ is an abelian ideal of codimension one, the Lie algebra structure of $\mathfrak{g}$ is completely determined by the action of $\operatorname{ad}_{\ecal}$ on $\mathfrak{a}$, which leaves $\mathfrak{a}$ invariant. We denote by $T: \mathfrak{a} \rightarrow \mathfrak{a}$ the linear map obtained by restricting both the domain and codomain of $\operatorname{ad}_{\ecal}$ to $\mathfrak{a}$. Formally, if $\pi_{\mathfrak{a}}^{\mathfrak{a} \oplus \mathfrak{a}^{\perp}}$ is the projection map from $\mathfrak{g}$ onto $\mathfrak{a}$, then $T$ is given by $\pi_{\mathfrak{a}}^{\mathfrak{a} \oplus \mathfrak{a}^{\perp}} \circ \operatorname{ad}_{\ecal}\!|_{\mathfrak{a}} $, and so $[\ecal, X] = TX$, for all $X \in \mathfrak{a}$.

Furthermore, since the restricted metric $\langle \cdot, \cdot \rangle|_{\mathfrak{a} \times \mathfrak{a}}$ is non-degenerate, any linear map $Q: \mathfrak{a} \rightarrow \mathfrak{a}$ admits a unique adjoint $Q^{\dagger}: \mathfrak{a} \rightarrow \mathfrak{a}$ defined by the relation $\langle QX, Y \rangle = \langle X, Q^{\dagger}Y \rangle$ for all $X, Y \in \mathfrak{a}$. Let $S := \frac{1}{2}(T + T^{\dagger})$ and $A := \frac{1}{2}(T - T^{\dagger})$ denote the symmetric and skew-symmetric parts of $T$, respectively.

Applying the Koszul formula, the Levi-Civita connection $\nabla$ of $(\mathfrak{g}, \langle \cdot, \cdot \rangle)$ is completely determined by its action on the subspace $\mathfrak{a}$ and the vector $\ecal$. For any $X, Y \in \mathfrak{a}$, it is straightforward to verify that:
$$
\nabla_{\ecal} \ecal = 0,  \quad
\nabla_{\ecal} Y = AY, \quad
\nabla_{X} {\ecal} = -SX, \quad
\nabla_{X} Y = \varepsilon \langle SX, Y \rangle {\ecal}.
$$
Consequently, the symmetric part of $\nabla$ is given by:
$$
U(\ecal,\ecal) = 0, \quad
U(\ecal,X) = -\frac{1}{2}T^{\dagger} X, \quad
U(X,Y) = \varepsilon \langle SX, Y \rangle \ecal.
$$


Assume that $U$ defines an associative product $\divideontimes$ on $\mathfrak{g}$. For any $W, X, Y \in \mathfrak{a}$, the associativity of $\divideontimes$ imposes the following constraints:

\begin{enumerate}[label=\Alph*.]
\item\label{nilpotencia} The condition $( \ecal \divideontimes \ecal)\divideontimes X = \ecal \divideontimes ( \ecal \divideontimes X)$ implies $T^{\dagger} \circ T^{\dagger} X = 0$. Hence, both $T^{\dagger}$ and $T$ are $2$-step nilpotent linear maps. Since $\mathfrak{g}$ is non-abelian, $T$ cannot be identically zero.

\item\label{isotropia} The condition $(\ecal \divideontimes X)\divideontimes Y = \ecal \divideontimes (X\divideontimes Y)$ implies $\langle S T^{\dagger} X , Y \rangle  = 0$. Since the restricted metric $\langle \cdot, \cdot \rangle|_{\mathfrak{a} \times \mathfrak{a}}$ is non-degenerate, it follows that $S \circ T^{\dagger} = 0$. Combined with the previous constraint, this yields $T \circ T^{\dagger} = 0$.

\item\label{normalidad} The condition $(W\divideontimes X)\divideontimes Y = W\divideontimes (X\divideontimes Y)$ implies $\langle S W , X \rangle T^{\dagger}Y = \langle S Y , X \rangle T^{\dagger} W$.
\end{enumerate}

Condition \ref{isotropia} says that the image of $T^{\dagger}$ is a \textit{totally isotropic subspace} of $(\mathfrak{a} , \langle \cdot, \cdot \rangle)$. Furthermore, computing the $(1,2)$-trace of the tensor $\theta(W,X,Y)=\langle W \divideontimes X , Y \rangle$ yields $\text{Tr}_{1,2}^{\langle\cdot,\cdot \rangle}\theta(Y) = \varepsilon \operatorname{Tr}(S) \langle Y, \ecal \rangle$. This trace vanishes identically because $\operatorname{Tr}(T)=0$, since $T$ is nilpotent. To fully exploit Condition \ref{normalidad}, we analyze the structure based on whether $S$ vanishes.

\subsection{The case $S\neq 0$}If $S\neq 0$, there exists $Y_0 \in \mathfrak{a}$ such that $S Y_0 \neq 0$. By the non-degeneracy of the restricted metric, we can find an $X_0 \in \mathfrak{a}$ satisfying $\langle S Y_0 , X_0 \rangle \neq 0$. Evaluating Condition \ref{normalidad} at these vectors and using the symmetry of $S$, we deduce that the image of $T^{\dagger}$ is one-dimensional. Thus, there exist non-zero vectors $u, \zcal \in \mathfrak{a}$ such that:
\begin{equation}\label{Ttranspuesta}
T^{\dagger} W = \langle W, \zcal \rangle u, \quad \text{for all } W \in \mathfrak{a}.
\end{equation}
Consequently, the image of $T^{\dagger}$ is spanned by $u$, which must lie in the kernel of $T^{\dagger}$ by the nilpotency condition. Furthermore, Condition \ref{isotropia} ensures that $u$ is an isotropic vector. By duality, Equation (\ref{Ttranspuesta}) leads directly to the following formula for $T$: for any $X, W \in \mathfrak{a}$, we have $\langle T W , X \rangle = \langle W , T^{\dagger} X \rangle = \langle X,\zcal  \rangle \langle W , u \rangle = \langle \langle W , u \rangle \zcal  , X \rangle    $. By non-degeneracy, this implies:
\begin{equation}\label{Toriginal}
T W = \langle W , u \rangle \zcal, \quad \text{for all } W \in \mathfrak{a}.
\end{equation}
This shows that the image of $T$ is the one-dimensional subspace of $\mathfrak{a}$ spanned by $\zcal$. We now return to Condition \ref{normalidad}. Substituting the expression for $T^{\dagger}$ from Equation (\ref{Ttranspuesta}) into Condition \ref{normalidad}, and canceling the non-zero vector $u$, we obtain:
$$
\langle S X  , W \rangle \langle Y, \zcal \rangle = \langle S X , Y \rangle \langle W,  \zcal \rangle,
$$
which is equivalent to:
$$
\langle \langle Y, \zcal \rangle S X  - \langle S X ,  Y \rangle \zcal , W \rangle = 0.
$$
Since this holds for all $W$, we conclude that the image of $S$ is spanned by $\zcal$. However, we know that $S = \frac{1}{2}(T + T^{\dagger})$. Because both $\operatorname{Im}(T)$ and $\operatorname{Im}(S)$ are spanned by $\zcal$, Equation (\ref{Ttranspuesta}) forces $u$ and $\zcal$ to be linearly dependent. Therefore, there exists a non-zero scalar $t \in \mathbb{R}$ such that $u = t \zcal$. This linear dependence immediately implies that $T$ is self-adjoint ($T = T^{\dagger}$). Indeed:
\begin{align*}
T^{\dagger} X & =  \langle X,  \zcal  \rangle  u &&   (\text{by Eq. (\ref{Ttranspuesta})})\\
              & =  \langle X,  \zcal  \rangle  t  \zcal  && (\text{since $ u = t  \zcal  $})\\
              & =  \langle X, t \zcal  \rangle  \zcal  && (\text{by bilinearity)}\\
              & =  \langle X, u \rangle   \zcal  &&  (\text{since $ u = t \zcal   $}) \\
              & =  T X                      &&  (\text{by Eq. (\ref{Toriginal})}).
 \end{align*}
We can now establish the structural decomposition of $\mathfrak{g}$. Let $n = \dim(\mathfrak{g})$. Since $u \in \mathfrak{a}$ is non-zero, we can choose a vector $\fcal \in \mathfrak{a}$ such that $\langle \fcal, u \rangle = 1$. By replacing $\fcal$ with $\fcal + s u$ for a suitable scalar $s$, we can assume without loss of generality that $\fcal$ is a null vector ($\langle \fcal, \fcal \rangle = 0$). This choice yields the Lie bracket $[\ecal, \fcal] = T(\fcal) = \zcal$.

The subspace $\mathcal{H} = \operatorname{span}\{\ecal, \fcal, \zcal\}$ forms a Lie subalgebra of $\mathfrak{g}$ isomorphic to $\mathfrak{h}_3$. Furthermore, the restriction of the metric $\langle \cdot, \cdot \rangle$ to $\mathcal{H}$ is non-degenerate and inherits a Lorentzian (or anti-Lorentzian) signature, depending on whether $\varepsilon = \langle \ecal, \ecal \rangle$ is $1$ or $-1$. Let $\mathcal{V} = \mathcal{H}^\perp$ be its orthogonal complement. Since $\mathcal{H}$ is non-degenerate, $\mathcal{V}$ is also non-degenerate, and $\mathcal{V} \subset \mathfrak{a}$. Because $\mathfrak{a}$ is abelian, the brackets $[\fcal, \mathcal{V}]$ and $[\zcal, \mathcal{V}]$ vanish identically. Moreover, for any $v \in \mathcal{V}$, the bracket $[\ecal, v] = T(v)$ is also zero; this follows directly from the explicit formula for $T$ in (\ref{Toriginal}) and the fact that $u$ is parallel to $\zcal$. Consequently, $\mathcal{V}$ is a central, non-degenerate ideal. We summarize this result in the following structural theorem:

\begin{theorem}
Let $(\mathfrak{g}, \langle \cdot, \cdot \rangle)$ be a pseudo-Riemannian almost-abelian Lie algebra admitting a non-degenerate abelian ideal $\mathfrak{a}$ of codimension one. Let $\ecal$ be a unit generator of $\mathfrak{a}^\perp$, and suppose the symmetric part of $\operatorname{ad}_{\ecal}$ with respect to $\langle \cdot, \cdot \rangle$ is non-zero. Then, the symmetric part of the Levi-Civita connection defines an associative structure on $\mathfrak{g}$ if and only if $(\mathfrak{g}, \langle \cdot, \cdot \rangle)$ is isometrically isomorphic to the direct product
\begin{equation*}
    (\mathfrak{h}_{3}, \langle \cdot, \cdot \rangle_1) \times (\mathbb{R}^{n-3}, \langle \cdot, \cdot \rangle_2),
\end{equation*}
where:
\begin{enumerate}
    \item $\mathfrak{h}_{3} = \operatorname{span}\{E, F, Z\}$ is the 3-dimensional Heisenberg algebra equipped with the pseudo-metric $\langle \cdot, \cdot \rangle_1$ defined by $\langle E, E \rangle_1 = \pm 1$, $\langle F, Z \rangle_1 = 1$, and all other inner products being zero.
    \item $\mathbb{R}^{n-3}$ is regarded as an abelian Lie algebra and $\langle \cdot, \cdot \rangle_2$ is a non-degenerate metric $\langle \cdot, \cdot \rangle_2$ such that the total signature of the product recovers the signature of the original metric $\langle \cdot, \cdot \rangle$.
\end{enumerate}
\end{theorem}

\subsection{The case $S=0$.}

In this scenario, the endomorphism $T$ is skew-symmetric with respect to the restricted metric; that is, $T = -T^{\dagger}$. Under this assumption, the associativity conditions derived from the $U$-tensor simplify significantly. As previously established in Condition \ref{nilpotencia}, the associativity requires $T^\dagger \circ T^\dagger = 0$, which implies $T^2 = 0$. Furthermore, Condition \ref{isotropia} states that $\operatorname{Im}(T)$ is a totally isotropic subspace of $(\mathfrak{a}, \langle \cdot, \cdot \rangle)$.

This case exhibits a remarkable geometric rigidity that is best understood through a \textit{Witt decomposition} of the ideal $\mathfrak{a}$ (see, for instance, \cite[Proposition 4.5.3]{Garling}). Let $W := \operatorname{Im}(T)$ and let $\{e_1, \dots, e_{2k}\}$ be a basis for $W$ (noting that its dimension must be even, as shown below). There exists a totally isotropic supplementary subspace $\widetilde{W} \subset \mathfrak{a}$ with a basis $\{\widetilde{e}_1, \dots, \widetilde{e}_{2k}\}$ such that $\langle e_i , \widetilde{e}_j \rangle = \delta_{i,j}$. Consequently, $W \oplus \widetilde{W}$ is a non-degenerate hyperbolic vector space, and we obtain a non-degenerate orthogonal complement with respect to $\mathfrak{a}$, given by $\mathcal{U} = (W \oplus \widetilde{W})^\perp$.

Since $T$ is $2$-step nilpotent, $T(W) = 0$, and since $\operatorname{Ker}(T) = (\operatorname{Im}(T^{\dagger}))^{\perp}$, it follows that $T(\mathcal{U})=0$. Therefore, the action of $T$ is entirely determined by its restriction to $\widetilde{W}$. Specifically, $T$ induces a linear map $\widetilde{\beta}: \widetilde{W} \rightarrow W$ by strictly restricting both its domain to $\widetilde{W}$ and its codomain to $W$. Because $\widetilde{W}$ and $W$ share the same dimension and $W = \operatorname{Im}(T)$, $\widetilde{\beta}$ results in a linear isomorphism. The fact that $T$ is a skew-symmetric map on $(\mathfrak{a}, \langle \cdot, \cdot \rangle|_{\mathfrak{a} \times \mathfrak{a}})$ translates to the fact that $\widetilde{\beta}$ defines a non-degenerate $2$-form $\widetilde{\omega}$ on $\widetilde{W}$, given by
\begin{equation}\label{betatransformation}
\widetilde{\omega}(\widetilde{X},\widetilde{Y}) = \langle \widetilde{\beta} \widetilde{X} , \widetilde{Y} \rangle,
\end{equation}
which implies that the dimension of $W$ must necessarily be even.

Let $\mathbb{T}W$ denote the vector space $W \oplus W^{\ast}$, where $W^{\ast}$ is the dual space of $W$, equipped with its \textit{canonical neutral metric} $\langle \! \langle \cdot, \cdot \rangle \!\rangle$ defined by
$$
\langle \! \langle  X + \xi ,  Y + \eta \rangle \!\rangle = \eta(X) + \xi(Y), \quad \mbox{ for all } X, Y \in W \mbox{ and for all } \xi, \eta \in W^{\ast}.
$$
The subspace $W \oplus \widetilde{W}$ with its restricted metric is isometric to $(\mathbb{T}W , \langle \! \langle \cdot, \cdot \rangle \!\rangle)$ via the linear isometry $I: W \oplus \widetilde{W} \rightarrow \mathbb{T}W$ defined by
$$
I(X) = \begin{cases} X, & \text{if } X \in W, \\ \langle X , \cdot \rangle|_{W}, & \text{if } X \in \widetilde{W}. \end{cases}
$$
Through this identification, we obtain a linear map $\beta: W^{\ast} \rightarrow W$ defined by $\beta(\alpha) := \widetilde{\beta}(I^{-1}(\alpha))$ for all $\alpha \in W^{\ast}$. By (\ref{betatransformation}), this map satisfies the condition
$$
\langle \! \langle \beta \eta, \xi \rangle \!\rangle = - \langle \! \langle \eta, \beta \xi \rangle \!\rangle \quad \text{for all } \xi, \eta \in W^{\ast}.
$$
Thus, utilizing the standard splitting $\mathbb{T} W = W \oplus W^{\ast}$, the map defines an \textit{infinitesimal $\beta$-transformation} of $\mathbb{T}W$ (see, for instance, \cite[Chapter 2]{Gualtieri}):
$$
D_{\beta} = \begin{pmatrix} 0 & \beta \\ 0 & 0 \end{pmatrix} \in \mathfrak{so}(\mathbb{T}W, \langle \! \langle \cdot, \cdot \rangle \!\rangle).
$$
Finally, the \textit{Witt index} of the metric provides the bound $\operatorname{dim}(W) \leq \min\{p, q\}$ if the signature of $\langle \cdot, \cdot \rangle|_{\mathfrak{a} \times \mathfrak{a}}$ is $(p,q)$. We summarize these elementary observations in the following structural theorem:

\begin{theorem}
Let $(\mathfrak{g}, \langle \cdot, \cdot \rangle)$ be a pseudo-Riemannian almost-abelian Lie algebra admitting a non-degenerate abelian ideal $\mathfrak{a}$ of codimension one. Let $(p,q)$ be the signature of $\langle \cdot, \cdot \rangle|_{\mathfrak{a} \times \mathfrak{a}}$ and let $\ecal$ be a unit generator of $\mathfrak{a}^\perp$. Suppose that $\operatorname{ad}_{\ecal}$ is a skew-symmetric linear map with respect to $\langle \cdot, \cdot \rangle$. Then, the symmetric part of the Levi-Civita connection defines an associative structure on $\mathfrak{g}$ if and only if $(\mathfrak{g}, \langle \cdot, \cdot \rangle)$ is isometrically isomorphic to the direct product
\begin{equation*}
    (\mathbb{R} \ecal \ltimes_{D_{\beta}} \mathbb{T} W , \langle \cdot, \cdot \rangle_1) \times (\mathbb{R}^{n-4k-1}, \langle \cdot, \cdot \rangle_2),
\end{equation*}
where:
\begin{enumerate}
    \item $W$ is a real vector space of dimension $2k$, with $2k \leq \min\{p,q\}$.
    \item $D_{\beta}$ is an infinitesimal $\beta$-transformation of $\mathbb{T}W = W \oplus W^{\ast}$ with $\beta$ being a non-degenerate bivector.
    \item $\langle \cdot, \cdot \rangle_1$ is a metric such that $\langle \ecal , \ecal \rangle_1 = \pm 1$, $\ecal$ is orthogonal to $\mathbb{T} W$, and the restriction $\langle \cdot, \cdot \rangle_1|_{\mathbb{T} W \times \mathbb{T} W}$ coincides with the canonical neutral metric of $\mathbb{T} W$.
    \item $\mathbb{R}^{n-4k-1}$ is regarded as an abelian Lie algebra endowed with a non-degenerate metric $\langle \cdot, \cdot \rangle_2$ such that the total signature of the product recovers the signature of the original metric $\langle \cdot, \cdot \rangle$.
\end{enumerate}
\end{theorem}

\section{Example: Two-step nilpotent Lie algebras}

Let us consider the case of a metric Lie algebra $(\mathfrak{n}, \langle \cdot, \cdot \rangle)$ where $\mathfrak{n}$ is 2-step nilpotent and the restriction of the metric to its center, $\mathfrak{z}$, is non-degenerate. This non-degeneracy condition allows us to decompose $\mathfrak{n}$ as an orthogonal direct sum $\mathfrak{n} = \mathcal{V} \oplus \mathfrak{z}$, where $\mathcal{V}$ is the non-degenerate orthogonal complement of $\mathfrak{z}$. In the spirit of the operators introduced by Aroldo Kaplan in his seminal work on $H$-type algebras \cite{Kaplan}, we can capture the relationship between the Lie bracket and the metric through a family of endomorphisms $J_Z: \mathcal{V} \to \mathcal{V}$, for $Z \in \mathfrak{z}$, implicitly defined by:\begin{eqnarray}\label{Kaplan}\langle J_Z E , F \rangle = \langle Z, [E,F] \rangle, \quad \mbox{for all } E,F \in \mathcal{V}.\end{eqnarray}Due to the skew-symmetry of the Lie bracket, these endomorphisms are skew-symmetric with respect to the restricted metric on $\mathcal{V}$ and completely determine the bracket structure. Using the Koszul formula to compute the Levi-Civita connection of $(\mathfrak{n}, \langle \cdot, \cdot \rangle)$, it is straightforward to verify that for all $E, F \in \mathcal{V}$ and $Z, \widetilde{Z} \in \mathfrak{z}$:
$$
\nabla_E F = \frac{1}{2}[E,F], \quad
\nabla_E Z = \nabla_Z E = -\tfrac{1}{2} J_{Z}E, \quad
\nabla_Z \widetilde{Z} = 0.
$$
Consequently, the symmetric part $U$ of $\nabla$ is given by:
$$
U(E,F) = 0, \quad
U(E,Z) = U(Z,E) = -\tfrac{1}{2} J_{Z}E, \quad
U(Z,\widetilde{Z}) = 0.
$$
Under these conditions, the $U$-tensor cannot be associative. Denoting the product $U(X,Y)$ by $X \divideontimes Y$, if we assume associativity, the condition $(Z \divideontimes \widetilde{Z}) \divideontimes E = Z \divideontimes ( \widetilde{Z} \divideontimes E)$ applied to elements $Z, \widetilde{Z} \in \mathfrak{z}$ and $E \in \mathcal{V}$ forces the relation:\begin{eqnarray}\label{nilpotencia2}J_{Z} \circ J_{\widetilde{Z}} = 0.\end{eqnarray}Let $\mathcal{W}$ be the subspace of $\mathcal{V}$ spanned by the images of the linear transformations $J_{Z}$ for all $Z \in \mathfrak{z}$. Equation (\ref{nilpotencia2}) implies that $\mathcal{W}$ is contained in the simultaneous kernel of all operators $J_{\widetilde{Z}}$. However, by the defining relation (\ref{Kaplan}), any element in this simultaneous kernel must commute with all of $\mathcal{V}$, which means it belongs to the center $\mathfrak{z}$. Thus, $\mathcal{W} \subset \mathfrak{z}$. Since $\mathcal{W}$ is a subspace of $\mathcal{V}$, and $\mathcal{V} \cap \mathfrak{z} = \{0\}$, we are forced to conclude that $\mathcal{W} = \{0\}$. This implies that all endomorphisms $J_Z$ are trivial, making the Lie algebra $\mathfrak{n}$ abelian.

\begin{proposition}Let $(\mathfrak{n}, \langle \cdot, \cdot \rangle)$ be a metric Lie algebra where $\mathfrak{n}$ is 2-step nilpotent and non-abelian. If the restriction of the metric to the center of $\mathfrak{n}$ is non-degenerate, then the $U$-tensor cannot be associative.
\end{proposition}

Let us now explore the opposite scenario. Motivated by the preceding result and the findings in Section \ref{almostabelian}, we consider the case where the center $\mathfrak{z}$ of $\mathfrak{n}$ is a maximally isotropic subspace of $(\mathfrak{n}, \langle \cdot, \cdot \rangle)$, and $\langle \cdot, \cdot \rangle$ is a \textit{split metric}. Using the Witt decomposition, there exists another totally isotropic subspace $\widetilde{\mathfrak{z}}$ such that $\mathfrak{n} = \mathfrak{z} \oplus \widetilde{\mathfrak{z}}$.

Under these conditions, it is straightforward to verify that $\mathfrak{z} \subseteq \operatorname{Ann}(\mathfrak{n}, \divideontimes)$. Indeed, for arbitrary $X,Y \in \mathfrak{n}$ and $Z \in \mathfrak{z}$, we have:
$$
2 \langle Z \divideontimes X , Y \rangle = \langle [Y, Z] , X \rangle + \langle [Y,X], Z \rangle.
$$
The first term vanishes because $Z \in \mathfrak{z}$, and the second term vanishes because $[Y,X] \in [\mathfrak{n},\mathfrak{n}] \subseteq \mathfrak{z}$ (since $\mathfrak{n}$ is 2-step nilpotent) and $\mathfrak{z}$ is a totally isotropic subspace.

Furthermore, the product $\widetilde{\mathfrak{z}} \divideontimes \widetilde{\mathfrak{z}}$ must be contained in $\mathfrak{z}$. Indeed, since $\mathfrak{z}$ is maximally isotropic, it coincides with its orthogonal complement ($\mathfrak{z} = \mathfrak{z}^\perp$). Thus, to prove that $\widetilde{\mathfrak{z}} \divideontimes \widetilde{\mathfrak{z}} \subseteq \mathfrak{z}$, it suffices to show that $\langle \widetilde{\mathfrak{z}} \divideontimes \widetilde{\mathfrak{z}} , \mathfrak{z} \rangle = 0$, which follows immediately from definition (\ref{definicionU}).

From these two elementary observations, it follows that the algebra $(\mathfrak{n}, \divideontimes)$ is associative, since the identity $X \divideontimes (Y \divideontimes Z) = 0 = (X \divideontimes Y) \divideontimes Z$ holds trivially for all $X,Y,Z \in \mathfrak{n}$. We summarize this result in the following proposition:

\begin{proposition}
Let $(\mathfrak{n}, \langle \cdot, \cdot \rangle)$ be a 2-step nilpotent metric Lie algebra equipped with a split metric. If the center of $\mathfrak{n}$ is a maximally isotropic subspace, then the associated $U$-tensor is associative. In particular, the commutative algebra $(\mathfrak{n}, \divideontimes)$ is nilpotent of step at most 2, satisfying $X \divideontimes (Y \divideontimes Z) = 0$ for all $X, Y, Z \in \mathfrak{n}$.
\end{proposition}

\section{Final remarks}

\begin{enumerate}
\item If $\mathfrak{g}$ is a simple Lie algebra, it is well-known that its Killing form defines a bi-invariant metric on $G$; and so $(\mathfrak{g},\divideontimes)$ is an abelian algebra. However, it is worth mentioning that $\mathfrak{g}$ can admit other pseudo-Riemannian metrics such that $(\mathfrak{g},\divideontimes)$ is a non-abelian commutative associative algebra. An explicit example is $\mathfrak{sl}(2,\mathbb{R})$ with the Lie bracket given by
$$
[{ e_1},{ e_2}]={ e_3}, \quad
[{ e_1},{ e_3}]={ e_1}, \quad
[{ e_2},{ e_3}]=-2{ e_1}-{ e_2}
$$
and the Lorentzian metric defined by $\langle e_1, e_1 \rangle = -1$, $\langle e_2, e_2 \rangle = 1$, and $\langle e_3, e_3 \rangle = 1$, with all other inner products being zero. The corresponding algebra $(\mathfrak{sl}(2,\mathbb{R}),\divideontimes)$ is given by:
\begin{eqnarray*}
e_1  \divideontimes e_1 = e_3, \quad
e_1  \divideontimes e_2 = -e_3, \quad
e_1  \divideontimes e_3 = \tfrac{1}{2}(e_1+e_2), \\
e_2  \divideontimes e_2 = e_3,  \quad
e_2  \divideontimes e_3 = -\tfrac{1}{2}(e_1+e_2), \quad
e_3  \divideontimes e_3 = 0.
\end{eqnarray*}
\item Some of the results presented here can be improved by replacing the associativity hypothesis on the $U$-tensor with a \textit{power-associativity hypothesis}, such as the one defining \textit{Jordan algebras}.
\end{enumerate}

\end{document}